\documentclass[10pt,a4paper]{amsart}
\usepackage[latin1]{inputenc}
\usepackage{amsmath}
\usepackage{amsfonts}
\usepackage{amssymb}
\usepackage{amsthm}
\usepackage{xypic}
\usepackage{color}
\usepackage{tikz-cd}

\newtheorem{pro}{Proposition}[section]
\newtheorem{teo}[pro]{Theorem}
\newtheorem{defi}[pro]{Definition}
\newtheorem{lem}[pro]{Lemma}
\newtheorem{cor}[pro]{Corollary}
\newtheorem{remark}[pro]{Remark}

\newtheorem{ex}[pro]{Example}

\newcommand{\pd}{{\mathrm{pd}}}

\newcommand{\modu}{{\mathrm{mod}}}

\newcommand{\proj}{{\mathrm{proj}}}

\newcommand{\findim}{{\mathrm{fin.dim}}}

\newcommand{\rk}{{\mathrm{rk}}}

\newcommand{\Phidim}{{\Phi\,\mathrm{dim}}}
\newcommand{\Psidim}{{\Psi\,\mathrm{dim}}}
\newcommand{\add}{{\mathrm{add}}}

\newcommand{\D}{\mathcal{D}}
\newcommand{\T}{\mathcal{T}}
\newcommand{\U}{\mathcal{U}}
\newcommand{\K}{\mathbb{K}}

\newcommand{\C}{\mathcal{C}}
\newcommand{\E}{\mathcal{E}}
\newcommand{\M}{\mathcal{M}}

\newcommand{\pj}{\mathcal{P}}

\author{Marcelo Lanzilotta$^{1}$ \\ Jos\'e Vivero$^{2}$}
\thanks{$^1$ Instituto de Matem\'atica y Estad\'istica ``Rafael Laguardia", Universidad de la Rep\'ublica, Uruguay.}
\thanks{$^2$ Department of Mathematics, London School of Economics and Political Science, United Kingdom.}

\begin{document}
\title{Generalised Lat-Igusa-Todorov Algebras and Morita Contexts}
\renewcommand{\shortauthors}{M. Lanzilotta, J. Vivero}
\email{marclan@fing.edu.uy}
\email{j.a.vivero-gonzalez@lse.ac.uk}
\maketitle

\begin{abstract} In this paper we define (special) GLIT classes and (special) GLIT algebras. We prove that GLIT algebras, which generalise Lat-Igusa-Todorov algebras, satisfy the finitistic dimension conjecture and give several properties and examples. In addition we show that special GLIT algebras are exactly those that have finite finitistic dimension. Lastly we study Morita algebras arising form a Morita context and give conditions for them to be (special) GLIT in terms of the algebras and bimodules used in their definition. As a consequence we obtain simple conditions for a triangular matrix algebra to be (special) GLIT and also prove that the tensor product of a GLIT $\mathbb{K}$-algebra with a path algebra of a finite quiver without oriented cycles is GLIT.\\


\end{abstract}

\section{Introduction}

This work is framed in the theory of representations of Artin algebras with a particular interest in the finitistic dimension ($\findim$) conjecture, which states that for a given Artin algebra $\Lambda$ there is a uniform bound for all projective dimensions of all finitely generated $\Lambda$-modules with finite projective dimension. This conjecture can be traced back to the work of H. Bass, particularly to the article \cite{Bass}. For a complete survey of the conjecture, up to 1995, we refer the reader to \cite{Z}.\\

In relation to the $\findim$ conjecture, K. Igusa and G. Todorov in \cite{IT} defined the now called Igusa-Todorov functions (IT functions) and used them to prove, among other important results, that all algebras with representation dimension smaller or equal than $3$ satisfy the $\findim$ conjecture. This result is quite interesting since the concept of representation dimension, defined by M. Auslander in \cite{Aus}, has been extensively studied in connection with several problems in Representation Theory. In \cite{W} J. Wei defined the concept of Igusa-Todorov algebra (IT algebra) and used the IT functions to prove they satisfy the $\findim$ conjecture. He provided an extensive list of IT algebras such as monomial algebras, special biserial algebras, tilted algebras and algebras with radical cube zero, among others. The question of whether all Artin algebras are IT was answered in 2016 by T. Conde in her Ph.D. thesis \cite{TC}. She used some results proved by R. Rouquier in \cite{R} to exhibit as a counterexample a family of algebras that are not IT, but do satisfy the $\findim$ conjecture, namely exterior algebras of vector spaces of dimension greater or equal than $3$.\\ 

In \cite{BLMV} the authors defined generalised IT functions, which gave way to a generalisation of the concept of Igusa-Todorov algebra. This new class of algebras, named Lat-Igusa-Todorov (LIT for short), satisfies the $\findim$ conjecture and strictly contains the class of IT algebras, since it includes all self-injective algebras. Note that the example provided by Conde is a family of self-injective algebras that are not IT. Among LIT algebras it is possible to find algebras that are not IT, neither self-injective. However not all Artin algebras are LIT. The definition of LIT algebra has the disadvantage of imposing the existence of a class of modules $\D\subseteq \modu\,\Lambda$ that is additively closed $(\add\,\D=\D)$, syzygy invariant $(\Omega(\D)\subseteq\D)$ and $\Phidim(\D)=0$, where $\Phi$ is the first IT function defined in \cite{IT}. Even though there are many well known classes with those properties, like Gorenstein projectives and modules that are left orthogonal to $\Lambda$, there are examples of algebras where a subcategory with the above properties can only be composed of projective modules. This can be seen in \cite{Barrios2}, where the authors find examples of algebras that are not LIT.\\ 

The main objective of this article is to give a way to extend the definition of LIT algebra, without losing the property of satisfying the $\findim$ conjecture. Our quest takes us to the concepts of (special) GLIT classes and (special) GLIT algebras. We prove that GLIT algebras satisfy the finitistic dimension conjecture and that the known examples of algebras that are not LIT, turn out to be GLIT. On the other hand, the concept of special GLIT algebra is actually equivalent to the algebra having finite finitistic dimension.\\

As a way to explore the scope of the developed theory, we study Morita rings arising from a Morita context. Morita contexts, also known as pre-equivalence data, have been introduced by H. Bass in (\cite{Bass2}, see also \cite{Cohn}) in relation to the Morita Theorems on equivalences of module categories. Let $T$ and $U$ be Artin algebras, a Morita context over $T,U$ is a $6$-tuple $\mathcal{M}=(T,N,M,U,\alpha,\beta)$, where $_{U}M_T$ is a $U$-$T$-bimodule, $_{T}N_U$ is a $T$-$U$-bimodule and the maps $\alpha: M\otimes_T N\rightarrow U$, $\beta: N\otimes_U M\rightarrow T$ are $U$-$U$ and $T$-$T$-bimodule homomorphisms respectively, satisfying the following associativity conditions, $\forall m,m'\in M, \forall n,n'\in N$:
$$\alpha(m\otimes n)m' = m\beta(n\otimes m')\ \ \text{and}\ \ \beta(n\otimes m)n' = n\alpha(m\otimes n').$$
Associated to any Morita context $\mathcal{M}$ as above, there is the Morita ring of $\mathcal{M}$, which incorporates all the information involved in the 6-tuple and is defined to be the formal $2\times 2$ matrix ring
$$\Lambda_{(\alpha,\beta)}(\mathcal{M})=\begin{pmatrix}
	T & N \\ M & U
\end{pmatrix},$$
where matrix multiplication is given by 
$$\begin{pmatrix}
	t & n \\ m & u
\end{pmatrix} \cdot \begin{pmatrix}
t' & n' \\ m' & u'
\end{pmatrix} = \begin{pmatrix}
tt' + \beta(n\otimes m') & tn' + nu' \\ mt' + um' & \alpha(m\otimes n') + uu'
\end{pmatrix}$$
We would like to note that the Morita ring of a Morita context should not to be confused with the notion of a (right or left) Morita ring appearing in Morita duality.\\

Since we are interested in Artin algebras, the following result characterises when a Morita ring is an Artin algebra.

\begin{pro}\cite[Prop. 2.2]{Psa}
	Let $\Lambda_{(\alpha,\beta)}(\mathcal{M})=\begin{pmatrix}
		T & N \\ M & U
	\end{pmatrix}$ be a Morita ring. Then $\Lambda_{(\alpha,\beta)}(\mathcal{M})$ is an Artin algebra if and only
	if there is a commutative Artin ring R such that T and U are Artin R-algebras and M and N are finitely
	generated over R which acts centrally both on M and N.
\end{pro}

Given a Morita ring arising from a Morita context, we prove that if $T,U$ are GLIT algebras, $_{U}M, M_T, _{T}N, N_U$ are projective modules and $M\otimes_TN=N\otimes_UM=0$, then the Morita ring is a GLIT Artin algebra. As a particular case we obtain that a triangular algebra of the form $\begin{pmatrix}
	T & 0 \\ M & U
\end{pmatrix}$ is GLIT provided the algebras in the diagonal are GLIT and $_{U}M, M_T$ are projectives. As a consequence we prove that the tensor product of a GLIT $\mathbb{K}$-algebra with a path algebra of a finite quiver without oriented cycles is GLIT.\\ 

A strong motivation for the study of tensor products is linked to the existence of Artin algebras of infinite $\Phi$-dimension (see \cite{Barrios} and \cite{Hanson}). In particular, the example provided by E. Hanson and K. Igusa in \cite{Hanson} is a $\mathbb{K}$-algebra that is a tensor product of two GLIT $\mathbb{K}$-algebras and in that particular case it is again a GLIT algebra, so it satisfies the finitistic dimension conjecture, even though its $\Phi$-dimension is infinite. A natural question that arises here is if the tensor product of two GLIT algebras is again GLIT. Our results provide a partial answer and establish a foundation for future research on the topic.\\

After this introduction, the reader will find Section 2 containing the necessary background material to understand the upcoming results. In Section 3 we give the definition and main properties of (special) GLIT classes and (special) GLIT algebras. Finally, Section 4 is dedicated to giving conditions for Morita rings, matrix rings and tensor products to be GLIT.

\section{Preliminaries}

In this section some required definitions and results are presented in order to make use of them in what follows. Unless otherwise stated, we are going to work with left modules over an Artin algebra.\\

{\sc Igusa-Todorov functions and algebras.} First we give the definition of the Igusa-Todorov function $\Phi$. We present an alternative way of defining the $\Phi$ function that is equivalent to the original definition given in \cite{IT}. For an Artin algebra $\Lambda$, we denote by $\proj\,\Lambda$ the full subcategory of $\modu\,\Lambda$ consisting of the projective modules. Define $K_{\proj\,\Lambda}$ to be the free abelian group generated by the set $\{[X]\}$ of isoclasses of indecomposable, non-projective f.g. $\Lambda$-modules. For every $X\in \modu\,\Lambda$, we can define the subgroup $\left\langle X\right\rangle $ as the free abelian group generated by the set of all isoclasses of indecomposable non-projective modules which are direct summands of $X$. We denote by $L$ the endomorphism of $K_{\proj\,\Lambda}$ defined by $L([X])=[\Omega(X)]$, where $\Omega$ denotes the syzygy operator. In this setting, the $\Phi$ function can be defined as follows.

\begin{defi}\label{defi:Phi}
Let $\Lambda$ be an Artin algebra. We define a function $\Phi:\modu\,\Lambda\rightarrow \mathbb{Z}_{\geq 0}$ in the following way:
$$\Phi(X):=\min\{n\in\mathbb{Z}_{\geq 0} : rk L^k(\left\langle X\right\rangle)=rk L^{k+1}(\left\langle X\right\rangle), \forall k\geq n\}.$$
\end{defi}

We recall the definition of the second Igusa-Todorov function $\Psi$. For any subcategoty $\mathcal{C}$ of $\modu\,\Lambda$ we set $$\findim(\mathcal{C}):=\sup\{ \pd(X)\ | \ X\in\mathcal{C}, \pd(X)<\infty \}.$$

\begin{defi}\label{defi:Psi}
	Let $\Lambda$ be an Artin algebra. We define a function $\Psi:\modu\,\Lambda\rightarrow \mathbb{Z}_{\geq 0}$ in the following way:
	$$\Psi(X) := \Phi(X) + \findim(\add\,(\Omega^{\Phi(X)}(X)).$$
\end{defi}

The following is a summary of properties of the Igusa-Todorov functions and their generalised versions. The reader can see \cite{IT}, \cite{HLM1}, \cite{HL} and \cite{BLMV} for a full picture.

\begin{pro}\cite[Lemma 3.4, Proposition 3.5]{HLM1}\label{pro:huard}
	Let $\Lambda$ be an Artin algebra. Then,  for every $X\in \modu\,\Lambda,$ we have:
	\begin{itemize}
		\item[(a)] $\Phi(X)\leq \Phi\left( \Omega (X)\right) +1;$
		\item[(b)] $\Psi(X)\leq \Psi\left( \Omega (X)\right) +1.$
	\end{itemize}
\end{pro}

\begin{defi}\label{defi:add_closed_and_sy_inv}
	Let $\Lambda$ be an Artin algebra, $\D$ a subcategory of $\modu\,\Lambda$ and $t\geq 1$ an integer. 
	\begin{itemize}
		\item[(1)] We say that $\D$ is \textbf{additively closed} if $\add\,\D=\D$.
		\item[(2)] We say that $\D$ is \textbf{$t$-syzygy invariant} if $\Omega^t\D\subseteq \D$.
	\end{itemize}
\end{defi}

In \cite{BLMV} the authors defined generalised Igusa-Todorov functions associated to a subcategory $\D$ of $\modu\,\Lambda$ that is additively closed and $1$-syzygy invariant. Said functions are denoted $\Phi_{[\D]}$ and $\Psi_{[\D]}$. It is pertinent to observe that if $\D$ is taken to be the subcategory consisting of projective modules, then the functions obtained are equal to the original Igusa-Todorov functions defined in \cite{IT}. We list now a few of the main properties of $\Phi_{[\D]}$ and $\Psi_{[\D]}$.

\begin{pro}\label{pro:properties_gen_IT_fun}
	Let $\Lambda$ be an Artin algebra and $\D$ a subcategory of $\modu\,\Lambda$ that is additively closed and syzygy invariant. Then, for $X,Y\in\modu\,\Lambda$ and $D\in\D$ we have
	\begin{itemize}
	    \item[(a)] $\Phi_{[\D]}(X\oplus D) = \Phi_{[\D]}(X).$ 
		\item[(b)] $\Phi_{[\D]}(X^s) = \Phi_{[\D]}(X).$
		\item[(c)] $\Phi_{[\D]}(X)\leq \Phi_{[\D]}(X\oplus Y).$
	\end{itemize}
\end{pro}

An important result that we will use later on is the following inequality.

\begin{teo}\cite[Theorem 3.5]{BLMV}\label{teo:Phi_D_inequality}
	Let $\Lambda$ be an Artin algebra and $\D$ a subcategory of $\modu\,\Lambda$ that is additively closed and syzygy invariant. Then for every $X\in\modu\,\Lambda$,
	$$\Phi(X) \leq \Phi_{[\D]}(X) + \Phidim(\D).$$
\end{teo}

Next we recall the definition of an Igusa-Todorov algebra given by J. Wei in \cite{W}.

\begin{defi}\label{defi:IT_algebra}
	Let $\Lambda$ be an Artin algebra. We say $\Lambda$ is $n$-Igusa-Todorov for some $n\geq 0$ if there is a module $V$ such that for every $M\in\modu\,\Lambda$ there is a short exact sequence $$0\rightarrow V_1\rightarrow V_0\rightarrow \Omega^nM\rightarrow 0,$$ with $V_0, V_1\in\add\,V.$ If we want to specify we say $\Lambda$ is $(n,V)$-Igusa-Todorov or $(n,V)$-IT for short.
\end{defi}

The relevance of IT algebras lies within the fact that they satisfy the $\findim$ conjecture and they include monomial algebras, special biserial algebras and radical cube algebras, among other relevant examples. However, all Artin algebras are not IT algebras, as was shown by T. Conde in her PhD thesis \cite{TC}. Using a previous result by R. Rouquier \cite[Corollary 4.4]{R} she was able to show that if $\K$ is an uncountable field and $V$ is a $\K$-vector space of dimension greater or equal than 3, then the exterior algebra $\bigwedge(V)$ is not an IT algebra.\\
On the other hand, exterior algebras are self-injective, so despite not being IT, their finitistic dimension is zero. In an attempt to generalise Wei's definition the authors in \cite{BLMV} defined new IT functions that gave way to the concept of LIT algebra, which generalises the definition of IT algebra. LIT algebras also satisfy the $\findim$ conjecture.

\begin{defi} \cite[Definition 5.1]{BLMV} An $n$-Lat-Igusa-Todorov algebra ($n$-LIT algebra, for short), where $n$ is a non-negative integer, is an Artin algebra $\Lambda$ satisfying the following two conditions:
	\begin{itemize}
		\item[(a)]  there is a class $\D\subseteq\modu\,\Lambda$ that is additively closed, $1$-syzygy invariant and $\Phidim\,(\D)=0;$
		\item[(b)] there exists $V\in\modu\,\Lambda$ satisfying that each $M\in \modu\,\Lambda$ admits an exact sequence 
		$$0\longrightarrow X_1\longrightarrow X_0\longrightarrow \Omega^nM\longrightarrow 0,$$ such that $X_1=V_1\oplus D_1$, $X_0=V_0\oplus D_0$, with $V_1,V_0\in \add\,V$ and $D_1,D_0\in \D.$
	\end{itemize}
	In case we need to specify the class $\D$ and the $\Lambda$-module $V,$ in the above definition, we say that $\Lambda$ is a $(n,V,\D)$-LIT algebra.
\end{defi} 

As we have said, LIT algebras include all the examples of non IT algebras that were known at the moment. This class of algebras also includes examples of non self-injective, non IT algebras. However, not all Artin algebras are LIT either, as was shown by M. Barrios and G. Mata in \cite{Barrios2}.\\

The main objective of this paper is to further generalise the concept of LIT algebra in order to create a new class of algebras so that all of the examples above are included and the $\findim$ conjecture continues to hold.\\

{\sc Morita contexts and rings.} Let $T$ and $U$ be Artin algebras, a Morita context over $T,U$ is a $6$-tuple $\mathcal{M}=(T,N,M,U,\alpha,\beta)$, where $_{U}M_T$ is a $U$-$T$-bimodule, $_{T}N_U$ is a $T$-$U$-bimodule and the maps $\alpha: M\otimes_T N\rightarrow U$, $\beta: N\otimes_U M\rightarrow T$ are $U$-$U$ and $T$-$T$-bimodule homomorphisms respectively, satisfying the following associativity conditions, $\forall m,m'\in M, \forall n,n'\in N$:
$$\alpha(m\otimes n)m' = m\beta(n\otimes m')\ \ \text{and}\ \ \beta(n\otimes m)n' = n\alpha(m\otimes n').$$
Associated to any Morita context $\mathcal{M}$ as above, there is the Morita ring of $\mathcal{M}$, which incorporates all the information involved in the 6-tuple and is defined to be the formal $2\times 2$ matrix ring
$$\Lambda_{(\alpha,\beta)}(\mathcal{M})=\begin{pmatrix}
	T & N \\ M & U
\end{pmatrix},$$
where matrix multiplication is given by 
$$\begin{pmatrix}
	t & n \\ m & u
\end{pmatrix} \cdot \begin{pmatrix}
	t' & n' \\ m' & u'
\end{pmatrix} = \begin{pmatrix}
	tt' + \beta(n\otimes m') & tn' + nu' \\ mt' + um' & \alpha(m\otimes n') + uu'
\end{pmatrix}$$

For simplicity we will write $\Lambda_{(\alpha,\beta)}$ instead of $\Lambda_{(\alpha,\beta)}(\mathcal{M})$. We say that a Morita ring is a Morita algebra if $\Lambda_{(\alpha,\beta)})$ is an Artin algebra.\\

We recall a very useful description for the module category of a Morita algebra. Any left $\Lambda_{(\alpha,\beta)}$-module can be seen as a 4-tuple $(A,B,f,g)$, where $A$ is a left $T$-module, $B$ is a left $U$-module, $f:M\otimes_T A\rightarrow B$ is a $U$-morphism and $g:N\otimes_U B\rightarrow A$ is a $T$-morphism. In addition, the following diagrams have to be commutative:

$$\xymatrix{N\otimes_U M\otimes_T A
	\ar[r]^(0.56){1_N\otimes f} 
	\ar[d]_{\beta\otimes 1_A} & 
	N\otimes_U B\ar[d]^{g}
	\\
	T\otimes_T A\ar[r]^{m_A} & A
} \ \ \ \ \ \
\xymatrix{M\otimes_T N\otimes_U B
	\ar[r]^(0.56){1_M\otimes g} 
	\ar[d]_{\alpha\otimes 1_B} & 
	M\otimes_T A\ar[d]^{f}
	\\
	U\otimes_U B\ar[r]^{m_B} & B,
}$$

where $m_A$ and $m_B$ are the isomorphisms given by the module action. We denote by $\beta_A:=m_A\circ (\beta\otimes 1_A)$ and $\alpha_B:=m_B\circ (\alpha\otimes 1_B)$.

A morphism $h:(A_1,B_1,f_1,g_1)\rightarrow (A_2,B_2,f_2,g_2)$ is a pair of morphisms $(h_1,h_2)$ such that $h_1\in Hom_T(A_1,A_2)$, $h_2\in Hom_U(B_1,B_2)$ and the following diagrams commute:
$$\xymatrix{
M\otimes_T A_1\ar[r]^{1\otimes h_1}\ar[d]^{f_1} & M\otimes_T A_2\ar[d]^{f_2} \\
B_1\ar[r]^{h_2} & B_2
}
\ \ \ \ \ \
\xymatrix{
	N\otimes_U B_1\ar[r]^{1\otimes h_2}\ar[d]^{g_1} & N\otimes_U B_2\ar[d]^{g_2}
	\\
	A_1\ar[r]^{h_1} & A_2 
}
$$

A 4-tuple $(A,B,f,g)$ is an indecomposable projective $\Lambda_{(\alpha,\beta)}$-module if and only if it has either the form $(P,M\otimes_T P, 1_{M\otimes_T P},\beta_P)$ or $(N\otimes_U Q,Q,\alpha_Q,1_{N\otimes_U Q})$, where $P$ and $Q$ are indecomposable projective modules over $T$ and $U$ respectively.\\
For further properties and more details we recommend \cite{Psa}.\\

The following result from \cite{IT-Morita} describes more precisely the syzygies of modules in a Morita algebra given some additional conditions on the Morita context. 

\begin{lem}\label{Syzygy_triang}
Let $\Lambda_{(0,0)}=\begin{pmatrix} T & N \\ M & U\end{pmatrix}$ be a Morita algebra sucht that $_{U}M, M_T, _{T}N$ and $N_U$ are projective modules. Then, for each $(A,B,f,g)\in\modu\,\Lambda$ we have $$\Omega_{\Lambda}(A,B,f,g)=(\Omega_TA, M\otimes_T P_A^{0}, 1\otimes i_0,0) \oplus (N\otimes_UQ_B^0,\Omega_UB, 0,1\otimes j_0),$$ where $P_A^{0}$ and $Q_B^0$ are the projective covers of $A$ and $B$ respectively. The maps $i_{0}:\Omega_TA\rightarrow P_A^{0}$ and $j_{0}:\Omega_UB\rightarrow Q_B^{0}$ are the canonical inclusions.
\end{lem}

We finish this section by making a very elementary but useful remark.

\begin{remark}\label{rmk:Syzygy_triang}
If, in addition to the hypotheses of Lemma \ref{Syzygy_triang}, we impose the condition $M\otimes_TN = N\otimes_UM=0$, then:
\begin{itemize}
	\item[(1)]$\Omega^n(A,B,f,g) = (\Omega^n(A),M\otimes P_{n-1}^A,1\otimes i_n,0) \oplus (N\otimes_UP_{n-1}^B, \Omega^n(B),0,1\otimes i_n)$, for any $n\geq 1$.
	\item[(2)] $\Omega(A, M\otimes_TP,f,g) = \Omega(A,0,0,0)$, for any $A\in\modu\,T$ and $P\in\proj\,T.$
	\item[(3)] $\Omega(N\otimes_UQ,B,f',g') = \Omega(0,B,0,0)$, for any $B\in\modu\,U$ and $Q\in\proj\,U.$
	\item[(4)] $\Omega^n(A,0,0,0) = \Omega(\Omega^{n-1}A,0,0,0),$ for all $n\geq 1.$
	\item[(5)] $\Omega^n(0,B,0,0) = \Omega(0,\Omega^{n-1}B,0,0),$ for all $n\geq 1.$
\end{itemize}

\end{remark}

\section{(special) GLIT classes and (special) GLIT algebras}

In this section we give the definition and main properties of (special) GLIT classes and (special) GLIT algebras. We start by looking at the question of  when a subcategory $\C$ of $\modu\,\Lambda$ with $\Phidim(\C)<\infty$ also satisfies that $\Psidim(\C)<\infty$.

\begin{pro}\label{pro:Psidim_finite}
	Let $\Lambda$ be an Artin algebra and $\C$ be a subcategory of $\modu\,\Lambda$ such that $\Phidim(\C)=n<\infty$. The following statements are satisfied. 
	\begin{itemize}
		\item[(i)] If $\findim(\add\,\Omega^t\C)<\infty \ \text{for some}\ t\geq n,$ then $\Psidim(\C)<\infty.$
		\item[(ii)] If $\C$ is closed under direct sums and $\Psidim(\C)<\infty$, then we have \\ $\findim(\add\,\Omega^n\C)<\infty.$
	\end{itemize}
\end{pro}

\begin{proof}
	(i) Let $C\in\C$, by definition we have $\Psi(C)=\Phi(C)+\findim(\add\,\Omega^{\Phi(C)}(C))$. Then there is $Z\mid \Omega^{\Phi(C)}(C)$ with $\pd(Z)<\infty$ such that $\Psi(C)=\Phi(C)+\pd(Z).$ Since $\Phi(C)\leq n\leq t$, there is $k\geq 0$ such that $\Phi(C)+k=t$. Because the syzygy operator commutes with direct sums we get $\Omega^k(Z)\mid \Omega^t(C)$. On the other hand we know that $\pd(Z)\leq\pd(\Omega^k(Z))+k$. Combining all of the above we obtain the following $$\Psi(C)=\Phi(C)+\pd(Z)\leq \Phi(C) +\pd(\Omega^k(Z))+k\leq \findim(\add\,\Omega^t\C) + t.$$ Since this can be done for any $C\in\C$, we get $$\Psidim(\C)\leq \findim(\add\,\Omega^t\C) + t<\infty.$$
	
	(ii) Let us denote $\Psidim(\C)=m$ and let $Z\in\add\,\Omega^n(\C)$ with $\pd(Z)<\infty$. Then there is $Z'$ an indecomposable summand of $Z$ and $C\in\C$ such that $Z'\mid \Omega^n(C)$ and $\pd(Z)=\pd(Z').$ Because $\Phidim(\C)=n$, there is $C'$ with $\Phi(C')=n$, then $\Phi(C\oplus C')=n$ and $$\Psi(C\oplus C')=n+\findim(\add\,\Omega^n(C\oplus C')).$$ This means that $\pd(Z)\leq m-n$ and since this can be done for any $Z$, we get $$\findim(\add\,\Omega^n\C)\leq m-n.$$
\end{proof}

\begin{cor}\label{cor:Psidim_finite1}
	Let $\Lambda$ be an Artin algebra and $\C$ be a subcategory of $\modu\,\Lambda$ closed under direct sums and such that $\Phidim(\C)=n<\infty$. Then,
	$$\Psidim(\C)<\infty \Longleftrightarrow \findim(\add\,\Omega^t\C)<\infty, \ \text{for some}\ t\geq n.$$
\end{cor}

\begin{cor}\label{cor:Psidim_finite2}
	Let $\Lambda$ be an Artin algebra and $\C$ be a subcategory of $\modu\,\Lambda$ such that $\Phidim(\C)=n<\infty$. $$\text{If}\ \Phidim(\Omega^t\C)<\infty \ \text{for some}\ t\geq n, \text{then} \ \Psidim(\C)<\infty.$$ Moreover, $\Psidim(\C)\leq \Phidim(\Omega^t\C)+t.$
\end{cor}

Let us show some examples of how we can use the previous results.

\begin{ex}\label{ej:CyD}
	\begin{itemize}
		\item[(1)] Let $\Lambda$ be an Artin algebra such that $\findim(\Lambda)<\infty$. Then every subcategory $\C$ of $\modu\,\Lambda$ with finite $\Phi$-dimension has also finite $\Psi$-dimension. This follows immediately from item (1) of Proposition \ref{pro:Psidim_finite}.
		\vspace{2mm}
		\item[(2)] Let $\Lambda$ be an Artin algebra and $\C$ a syzygy invariant subcategory of $\modu\,\Lambda$ such that $\Phidim(\C)=n<\infty$, then $\Psidim(\C)<\infty$. It suffices to apply Corollary \ref{cor:Psidim_finite2}, since $\Phidim(\Omega^n\C) \leq \Phidim(\C) < \infty$.
		\vspace{2mm}
		\item[(3)] Let $\D$ be a subcategory of $\modu\,\Lambda$ such that $\D$ is additively closed, syzygy invariant and $\Phidim(\D)=d<\infty$. Then for any $V\in\modu\,\Lambda$ the class $\C:=\D\oplus \add\,V$ satisfies $\Psidim(\C)<\infty$. In order to show this we will use Corollary \ref{cor:Psidim_finite2}.\\ 
		
		For an element $C=D_0\oplus V_0\in \D\oplus \add\,V$ we have 
		\begin{align*}
		\Phi(C) = \Phi(D_0\oplus V_0) & \leq \Phi_{[\D]}(D_0\oplus V_0)+d = \\ & =\Phi_{[\D]}(V_0)+d\leq \Phi_{[\D]}(V)+d.
		\end{align*}
		In order to obtain that bound we have used Theorem \ref{teo:Phi_D_inequality} and Proposition \ref{pro:properties_gen_IT_fun}. Hence, $\Phidim(\C)=n\leq \Phi_{[\D]}(V)+d<\infty.$\\
		
		Now, take $t\geq n$ and let $X\in\Omega^t(\C)$, then $X=\Omega^tD_1\oplus\Omega^tV_1$ and
		\begin{align*}
		\Phi(X)=\Phi(\Omega^tD_1\oplus\Omega^tV_1) & \leq \Phi_{[\D]}(\Omega^tD_1\oplus\Omega^tV_1)+d = \\ &=\Phi_{[\D]}(\Omega^tV_1)+d\leq \Phi_{[\D]}(\Omega^tV)+d.
		\end{align*}
		Here we have used the same inequalities as above, plus the fact that $\D$ is syzygy invariant. Therefore, $\Phidim(\Omega^t(\C))=k\leq \Phi_{[\D]}(\Omega^tV)+d<\infty.$\\
		In particular, taking $t=n$ we can apply Corollary \ref{cor:Psidim_finite2} to get $$\Psidim(\C)\leq \Phi_{[\D]}(V) + \Phi_{[\D]}(\Omega^nV) + 2d<\infty.$$
	\end{itemize}
\end{ex}

Now we examine the question of whether item $(3)$ of the previous example still holds if $\D$ is a class such that $\Phidim(\D)<\infty$, it is additively closed and $t$-syzygy invariant for some $t\geq 1$. The following results will lead to Lemma \ref{lem:previo_findim_finita}, where we prove that it is indeed possible.

\begin{lem}\label{lem:D_plus_indecomposable}
	Let $V\in\modu\,\Lambda$ be an indecomposable module and $\D\subseteq\modu\,\Lambda$ be a class such that $\add\,\D=\D$ and $\Phidim(\D)=n<\infty$. Then $\Phidim(\D\oplus\add\,V)<\infty.$
\end{lem}

\begin{proof}
	First we observe that $\Phidim(\D\oplus\add\,V) = \Phidim(\D \oplus V)$, so we will prove that the second term is bounded. If $\pd(V)<\infty,$ then $$\Phidim(\D\oplus V)=\max\{\Phidim(\D),\pd(V)\}<\infty.$$
    Assume now that $\pd(V)=\infty$ and that there exists $\tilde{D}\in\D$ such that $\Phi(\tilde{D}\oplus V)=m>n.$ We affirm that $\Phidim(\D\oplus V)\leq m.$ To prove the affirmation let us suppose that it is not true, so assume there is some $D'\in\D$ such that $\Phi(D'\oplus V)=r>m>n.$ This means that $s:=\Phi(D'\oplus \tilde{D}\oplus V)\geq r>m>n.$ By the definition of the $\Phi$ function there must be a non trivial relation (with integer coefficients) of the form 
	\begin{equation}\label{eq:level_s_relation}
		M_1[\Omega^sV] = \sum_{j=1}N_j[\Omega^s\tilde{D}_j] + \sum_{k=1}R_k[\Omega^sD'_k],
	\end{equation}
	where $\{\tilde{D}_j\}$ and $\{D'_k\}$ are the set of nonisomorphic, indecomposable, nonprojective direct summands of $\tilde{D}$ and $D'.$ This relation cannot occur at the $s-1$ syzygy level, also the coefficient $M_1$ is nonzero because $s>n$, and at least one of the coefficients $\{R_k\}$ is nonzero, since $s>m$.\\
	In a similar way, using that $\Phi(\tilde{D}\oplus V)=m<s$ there must be a non trivial relation of the form 
	\begin{equation}\label{eq:level_m_relation}
		M'_1[\Omega^mV] = \sum_{j=1}N'_j[\Omega^m\tilde{D}_j].
	\end{equation}
	This relation cannot occur at the $m-1$ syzygy level and similarly to equation (\ref{eq:level_s_relation}), the coefficient $M_1'$ is nonzero because $m>n$ and at least one of the coefficients $\{N_j'\}$ is nonzero since $V$ has infinite projective dimension. Because of the fact that $m<s$ equation (\ref{eq:level_m_relation}) will remain the same at the $s$ syzygy level (and at the $s-1$ syzygy level), so we get 
	\begin{equation}\label{eq:level_m_s_relation}
		M'_1[\Omega^sV] = \sum_{j=1}N'_j[\Omega^s\tilde{D}_j].
	\end{equation}
	We multiply equation (\ref{eq:level_s_relation}) by $M_1'$ and equation (\ref{eq:level_m_s_relation}) by $M_1$, then we get the following equality:
	\begin{equation*}
		\sum_{j=1}(M_1'N_j-M_1N'_j)[\Omega^s\tilde{D}_j] + \sum_{k=1}M_1'R_k[\Omega^sD'_k] = 0.
	\end{equation*}
	This is a nontrivial relation because for some $k$ the coefficient $M_1'R_k\neq 0$, but since $s>n$, this relation has to occur at the $s-1$ syzygy level. We can rewrite it as follows:
	\begin{equation*}
		\sum_{j=1}M_1'N_j[\Omega^{s-1}\tilde{D}_j] + \sum_{k=1}M_1'R_k[\Omega^{s-1}D'_k] = \sum_{j=1}M_1N'_j[\Omega^{s-1}\tilde{D}_j].
	\end{equation*}
	At this point we use that equation (\ref{eq:level_m_s_relation}) holds at the $s-1$ syzygy level. This means that the previous equation can be rewritten as
	\begin{equation*}
		\sum_{j=1}M_1'N_j[\Omega^{s-1}\tilde{D}_j] + \sum_{k=1}M_1'R_k[\Omega^{s-1}D'_k] = M_1M_1'[\Omega^{s-1}V].
	\end{equation*}
	Now we divide by $M_1'$ and obtain 
	\begin{equation}\label{eq:level_s-1_relation}
		\sum_{j=1}N_j[\Omega^{s-1}\tilde{D}_j] + \sum_{k=1}R_k[\Omega^{s-1}D'_k] = M_1[\Omega^{s-1}V].
	\end{equation}
	Note that equation (\ref{eq:level_s-1_relation}) is the same as equation (\ref{eq:level_s_relation}) but at the $s-1$ level of syzygy, which is a contradiction. This shows that the affirmation $\Phidim(\D\oplus\add\,V)\leq m$ is true and this finishes the proof.
\end{proof}

Next, we use induction to prove that the above result is valid even if $V$ is not indecomposable.

\begin{lem}\label{lem:D_plus_V}
	Let $V\in\modu\,\Lambda$ be any module and $\D\subseteq\modu\,\Lambda$ be an additively closed class such that $\Phidim(\D)=n<\infty$. Then $\Phidim(\D\oplus\add\,V)<\infty.$
\end{lem}

\begin{proof}
	We use induction in the number of indecomposable summands of $V$. First if $V$ is indecomposable, then by the previous lemma we have the result. Assume that it is true for $s-1$ indecomposable direct summands and let $V=\oplus_{i=1}^sV_i$ be the Krull-Schmidt decomposition of $V$. It is not hard to see that $\D\oplus \add\,V=\D\oplus \add\,(\oplus_{i=1}^{s-1}V_i) \oplus \add\,V_s$. By the induction hypothesis $\Phidim(\D\oplus \add\,(\oplus_{i=1}^{s-1}V_i))<\infty$ and since $V_s$ is indecomposable, we can apply Lemma \ref{lem:D_plus_indecomposable} to obtain that $\Phidim(\D\oplus \add\,V)<\infty.$
\end{proof}

\begin{lem}\label{lem:previo_findim_finita}
	Let $\D\subseteq\modu\,\Lambda$ be such that $\Phidim(\D)<\infty$, $\D$ is additively closed and it is $t$-syzygy invariant for some $t\geq 1$. Let $V\in\modu\,\Lambda$, then
	\begin{itemize}
		\item[(i)] $\Phidim(\D\oplus\add\,V)=m<\infty.$
		\item[(ii)] For any $k\geq 1$, $\Phidim(\Omega^k(\D\oplus\add\,V))<\infty.$
		\item[(iii)] $\Psidim(\D\oplus\add\,V)<\infty.$
	\end{itemize}
\end{lem}

\begin{proof}
	The proof of $(i)$ is immediate from Lemma \ref{lem:D_plus_V}.\\
	For $(ii)$, let $D_0\oplus V_0\in \D\oplus\add\,V$ we get $$\Phi(\Omega^k(D_0\oplus V_0))\leq \Phi(\Omega^{kt}(D_0\oplus V_0)) + k(t-1) = \Phi(\Omega^{kt}(D_0)\oplus \Omega^{kt}(V_0))) + k(t-1).$$
	Since $\D$ is $t$-syzygy invariant, then it is also $kt$-syzygy invariant. Hence we obtain 
	$$\Phidim(\Omega^k(\D\oplus\add\,V)) \leq \Phidim(\D\oplus \add\,\Omega^{kt}V) + k(t-1)<\infty.$$
	Note that in the particular case $k=qt$, 
	$$\Phidim(\Omega^k(\D\oplus\add\,V))\leq \Phidim(\D\oplus \add\,\Omega^{k}V)<\infty.$$
	The proof of $(iii)$ follows from the previous items in combination with Corollary \ref{cor:Psidim_finite2}, taking $k:=tm$. Note that $m=0$ implies $\Psidim(\D\oplus\add\,V)=0.$
\end{proof}

At this point we are ready to give the definition of a GLIT class.

\begin{defi}\label{defi:GLIT_class} Let $\Lambda$ be an Artin algebra and $\C$ be a subcategory of $\modu\,\Lambda$. We say $\C$ is an $n$\textbf{-GLIT} class, where $n\geq 0$, if the following two conditions hold:
	\begin{itemize}
		\item[(a)]  there is a class $\D\subseteq\modu\,\Lambda$ that is additively closed,  $\Omega^t(\D)\subseteq\D$ for some $t\geq 1$ and $\Phidim\,(\D)<\infty;$
		\item[(b)] there exists $V\in\modu\,\Lambda$ satisfying that each $C\in\C$ admits an exact sequence 
		$$0\longrightarrow X_1\longrightarrow X_0\longrightarrow \Omega^nC\longrightarrow 0,$$ such that $X_1=V_1\oplus D_1$, $X_0=V_0\oplus D_0$, with $V_1,V_0\in \add\,V$ and $D_1,D_0\in \D.$
	\end{itemize}
	In case we need to specify all the parameters in the above definition, we say that $\C$ is an $(n,t,V,\D)$-GLIT class.
\end{defi}

\begin{defi}\label{GLIT_algebra}
	We say that an Artin algebra $\Lambda$ is GLIT if $\modu\,\Lambda$ is an $(n,t,V,\D)$-GLIT class.
\end{defi}

\begin{remark}\label{rmk_GLIT class_condition(1)}
	If $\C$ is an $(n,t,V,\D)$-GLIT class, then $\D\oplus\add\,V$ satisfies items (i), (ii) and (iii) of Lemma \ref{lem:previo_findim_finita}.
\end{remark}

The following remark will be used in the proof of Theorem \ref{teo:Morita_GLIT}.

\begin{remark}\label{addD}
	It is possible to obtain a class $\D$, as in item $(a)$ of Definition \ref{defi:GLIT_class}, if what we have instead is a class $\D'$ that is closed under direct sums, is $t$-syzygy invariant and $\Phidim\,(\D')<\infty.$ Namely the class $\D:=\add(\D')$ satisfies all the conditions in $(a)$:
	\begin{itemize}
		\item It is clear from the definition that $\D$ is additively closed.
		\item Take $X\in \D$, then there exists $Y\in\D$ such that $X\oplus Y\in \D'$. Because $\D'$ is $t$-syzygy invariant and the syzygy operator commutes with direct sums we have $\Omega^t(X)\oplus\Omega^t(Y)\in \D'$, hence $\Omega^t(X)\in \add(\D')=\D,$ so we get $\Omega^t(\D)\subseteq \D.$
		\item Finally, let $X\in \D$ and $Y\in\D$ be such that $X\oplus Y\in \D'$. From the properties of the $\Phi$ function, we have $\Phi(X)\leq\Phi(X\oplus Y)\leq\Phidim(\D')$, hence $\Phidim\,(\D)=\Phidim\,(\D')<\infty$.
	\end{itemize} 
\end{remark}

\begin{pro}\label{pro:GLIT_bigger_integers}
	If $\C$ is $(n,t,V,\D)$-GLIT, then $\C$ is $(m,t,\tilde{V},\tilde{\D})$-GLIT for every $m\geq n$. 
\end{pro}

\begin{proof}
	Let us define $\tilde{V}:=\Omega^{m-n}V\oplus \Lambda$ and $\tilde{\D}:=\add\,\Omega^{m-n}\D$. We can check that item (b) of Definition \ref{defi:GLIT_class} holds by applying the syzygy operator $m-n$ times to the short exact sequences given by the fact that $\C$ is GLIT.\\
	To see that item (a) holds it suffices, because of the previous remark, to prove that $\Phidim(\Omega^{m-n}\D)<\infty$, since the rest of the conditions are easily verified. For any $Z\in\Omega^{m-n}\D$ we have that $\Omega^{(m-n)(t-1)}(Z)\in \D$, then by Proposition \ref{pro:huard} we get $$\Phi(Z)\leq \Phi(\Omega^{(m-n)(t-1)}(Z)) + (m-n)(t-1)\leq d + (m-n)(t-1)<\infty,$$ where $d=\Phidim(\D).$
\end{proof}

Now we can prove that GLIT algebras satisfy the finitistic dimension conjecture.

\begin{teo}\label{teo:GLIT_findim}
	Let $\Lambda$ be an Artin algebra and $\C\subseteq\modu\,\Lambda$ be an $(n,t,V,\D)$-GLIT class. Then $\findim(\C)<\infty.$ In particular, if $\Lambda$ is a GLIT algebra, then it satisfies the $\findim$ conjecture.
\end{teo}

\begin{proof}
	Let $C\in\C$ with $\pd(C)<\infty$. Then it is known that $$\pd(C)\leq\pd(\Omega^nC)+n.$$
	Now, let us use that $\C$ is a GLIT class and obtain a s.e.s of the form $$0\longrightarrow X_1\longrightarrow X_0\longrightarrow \Omega^nC\longrightarrow 0,$$ such that $X_1=V_1\oplus D_1$, $X_0=V_0\oplus D_0$, with $V_1,V_0\in \add\,V$ and $D_1,D_0\in \D.$ By the Igusa-Todorov inequality \cite[Theorem 4]{IT}, we have $$\pd(\Omega^nC)\leq\Psi(X_1\oplus X_0)+1\leq\Psidim(\D\oplus \add\,V)+1.$$ By Lemma \ref{lem:previo_findim_finita} $(iii)$ we know that $\Psidim(\D\oplus \add\,V)<\infty$, so in conclusion $$\findim(\C)\leq\Psidim(\D\oplus \add\,V)+n+1<\infty.$$
\end{proof}

\begin{ex}
	\begin{itemize}
		\item[(1)] If $\Lambda$ is $(n,V)$ Igusa-Todorov, then it is $(n,1,V,0)$-GLIT.
		\item[(2)] If $\Lambda$ is $(n,V,\D)$ Lat-Igusa-Todorov, then it is $(n,1,V,\D)$-GLIT.
		\item[(3)] If $\Phidim(\Lambda)<\infty$, then $\Lambda$ is $(0,1,0,\modu\,\Lambda)$-GLIT.
	\end{itemize}
\end{ex}

The next proposition gives necessary and sufficient conditions for a GLIT class to have finite finitistic dimension. In particular, it gives several characterisations of the finitistic dimension conjecture.

\begin{pro}\label{teo:special_GLIT_findim}
	Let $\Lambda$ be an Artin algebra and $\C$ a subcategory of $\modu\,\Lambda$. The following are equivalent:
	\begin{itemize}
		\item[(i)] $\findim(\C)<\infty.$
		\item[(ii)] $\pj^{<\infty}(\C)$ is a $0$-GLIT class.
		\item[(iii)] $\pj^{<\infty}(\C)$ is an $n$-GLIT class for every $n\geq 0$.
		\item[(iv)] There is some $m\geq 0$ and $\E\subseteq\modu\,\Lambda$ with $\Psidim(\E)<\infty$ and closed under direct sums such that for each $C\in\pj^{<\infty}(\C)$ there is a s.e.s. of the form $$0\rightarrow E_1\rightarrow E_0\rightarrow \Omega^mC\rightarrow 0, \ \text{with}\ E_1,E_0\in\E.$$
	\end{itemize}
\end{pro}

\begin{proof}\
\begin{itemize}
	\item[$(i)\Rightarrow (ii)$] Take $n=0, t=\findim(\C), V=0, \D=\pj^{<\infty}(\C)$. Note that $\Phidim(\pj^{<\infty}(\C))=\findim(\C)<\infty.$
	
	\item[$(ii)\Rightarrow (iii)$] Follows directly from Proposition \ref{pro:GLIT_bigger_integers}.
	
	\item[$(iii)\Rightarrow (iv)$] From Remark \ref{rmk_GLIT class_condition(1)} we obtain $\E:=\D\oplus\add\,V$ satisfies $\Psidim(\E)<\infty$ and it is closed under direct sums, so by the definition of GLIT class we get the desired short exact sequences.
	
	\item[$(iv)\Rightarrow (i)$] Analogous to the proof of Theorem \ref{teo:GLIT_findim}.
\end{itemize}
\end{proof}

\begin{defi}\label{defi:special_GLIT}
	Let $\Lambda$ be an Artin algebra and $\C$ a subcategory of $\modu\,\Lambda$. We say that $\C$ is \textbf{special GLIT} if it satisfies one of the equivalent conditions above. In particular, $\Lambda$ is a special GLIT algebra if $\modu\,\Lambda$ is a special GLIT class.
\end{defi}

We end this section with the following characterisation for the finitistic dimension conjecture. The proof follows from the above proposition and it is left to the reader.

\begin{teo}
	Let $\Lambda$ be an Artin algebra. Then, $\Lambda$ satisfies the finitistic dimension conjecture if and only if it is special GLIT.
\end{teo}

\section{Morita algebras: an application}

In this section we give conditions for a Morita algebra to be GLIT or special GLIT in terms of the Morita context.\\
Let $\mathcal{M}=(T, N, M, U, \alpha, \beta)$ be a Morita context such that:
\begin{itemize} 
	\item $_{U}M, M_T, _{T}N, N_U$ are projective modules.
	\item $M\otimes_T N = N\otimes_U M =0.$  
\end{itemize}
Let $\Lambda_{\mathcal{M}}:= \begin{pmatrix}
T & N \\
M & U 
\end{pmatrix}$ be the Morita algebra associated to the Morita context. In this section we prove that, with the above conditions, $\Lambda_{\mathcal{M}}$ is GLIT if and only if $T$ and $U$ are GLIT. This provides new examples of GLIT algebras and as an application we give a partial answer to the question of whether the tensor product of GLIT algebras is again GLIT.\\

From now on $\mathcal{M}$ will be taken to be a Morita context with the above conditions and $\Lambda_{\mathcal{M}}$ the Morita algebra associated to it. In the next lemma and proposition we establish a relationship between $\Phi(A,M\otimes_TP,f,g)$ and $\Phi(A)$, where $P\in\proj\,T$.

\begin{lem}\label{lem:relation_A_and_(A,P,f)}
	Let $(A,M\otimes_TP,f,g)\in\modu\,\Lambda_{\mathcal{M}}$ such that $P\in\proj\,T$. Denote the set of indecomposable, pairwise non isomorphic, direct summands of $(A,M\otimes_TP,f,g)$ by $\{(A_i,Q_i,f_i,g_i)\}_{i=1}^r$. If for some $n\geq 0$ and $N_i\in\mathbb{Z}_{\geq 0}, \forall i=1,\cdots, r$ we have a linear combination of the form:
	$$N_1[\Omega^nA_1] + \cdots + N_s[\Omega^nA_s] = N_{s+1}[\Omega^nA_{s+1}] + \cdots + N_r[\Omega^nA_r],$$
	then there is also a linear combination of the form:
	\begin{align*}
		& N_1[\Omega^{n+1}(A_1,Q_1,f_1,g_1)] + \cdots + N_s[\Omega^{n+1}(A_s,Q_s,f_s,g_s)] = \\
		& = N_{s+1}[\Omega^{n+1}(A_{s+1},Q_{s+1},f_{s+1},g_{s+1})] + \cdots + N_r[\Omega^{n+1}(A_r,Q_r,f_r,g_r)].
	\end{align*}
\end{lem}

\begin{proof}
Suppose that for some $n\geq 0$ we have a linear combination as follows: 
\begin{equation}\label{eq:lin_comb_A_i}
	N_1[\Omega^nA_1] + \cdots + N_s[\Omega^nA_s] = N_{s+1}[\Omega^nA_{s+1}] + \cdots + N_r[\Omega^nA_r],
\end{equation}
where $N_i\in\mathbb{Z}_{\geq 0}, \forall i=1,\cdots, r$. This means that there exist $P',P''\in\proj\,T$ such that $$N_1\Omega^nA_1 \oplus \cdots \oplus N_s\Omega^nA_s \oplus P' \simeq N_{s+1}\Omega^nA_{s+1} \oplus \cdots \oplus N_r\Omega^nA_r \oplus P''.$$ For each $i$ decompose $\Omega^nA_i=M_i\oplus P_i,$ in such a way that $P_i\in\proj\,T$ and $M_i$ does not have any projective direct summands. The previous isomorphism can be rewritten as $$N_1M_1 \oplus \cdots \oplus N_sM_s \oplus N_1P_1 \oplus \cdots \oplus P' \simeq N_{s+1}M_{s+1} \oplus \cdots \oplus N_rM_r \oplus N_{s+1}P_{s+1}\oplus \cdots \oplus P''.$$ From the Krull-Schmidt theorem we then have 
\begin{equation}\label{eq:iso_M_i}
	N_1M_1 \oplus \cdots \oplus N_sM_s \simeq N_{s+1}M_{s+1} \oplus \cdots \oplus N_rM_r.
\end{equation}
\begin{equation}\label{eq:iso_P_i}
	N_1P_1 \oplus \cdots \oplus N_sP_s \oplus P'\simeq N_{s+1}P_{s+1}\oplus \cdots \oplus N_rP_r \oplus P''.
\end{equation}
If we apply the syzygy operator to (\ref{eq:iso_M_i}) we get $$N_1\Omega M_1 \oplus \cdots \oplus N_s\Omega M_s \simeq N_{s+1}\Omega M_{s+1} \oplus \cdots \oplus N_r\Omega M_r.$$ Furthermore we have the equalities $\Omega M_i = \Omega^{n+1}A_i$ and $P_n^{A_i}=P_0^{M_i}\oplus P_i,$ for each $i=1,\cdots,r.$ From (\ref{eq:iso_M_i}) we have the following commutative diagram:
$$\xymatrix{N_1\Omega M_1 \oplus \cdots \oplus N_s\Omega M_s\ar[r]^i\ar[d]^{\simeq} & N_1P_0^{M_1} \oplus \cdots \oplus N_sP_0^{M_s}\ar[d]^{\simeq} \\
	N_{s+1}\Omega M_{s+1} \oplus \cdots \oplus N_r\Omega M_r\ar[r]^i & N_{s+1}P_0^{M_{s+1}} \oplus \cdots \oplus N_rP_0^{M_r}}$$
Using (\ref{eq:iso_P_i}) we can transform the previous diagram into the next one:
$$\xymatrix{N_1\Omega M_1 \oplus \cdots \oplus N_s\Omega M_s\ar[r]^{(i,0)}\ar[d]^{\simeq} & Y\oplus N_1P_1\oplus\cdots \oplus N_sP_s \oplus P' \ar[d]^{\simeq} \\
	N_{s+1}\Omega M_{s+1} \oplus \cdots \oplus N_r\Omega M_r\ar[r]^{(i,0)} & Z\oplus N_{s+1}P_{s+1}\oplus \cdots \oplus N_rP_r \oplus P''}$$ where $Y=N_1P_0^{M_1} \oplus \cdots \oplus N_sP_0^{M_s}$ and $Z=N_{s+1}P_0^{M_{s+1}} \oplus \cdots \oplus N_rP_0^{M_r}$. Now, using that $P_n^{A_i}=P_0^{M_i}\oplus P_i$ we get 
$$\xymatrix{N_1\Omega M_1 \oplus \cdots \oplus N_s\Omega M_s\ar[r]^{i'}\ar[d]^{\simeq} & N_1P_n^{A_1}\oplus\cdots \oplus N_sP_n^{A_s} \oplus P' \ar[d]^{\simeq} \\
	N_{s+1}\Omega M_{s+1} \oplus \cdots \oplus N_r\Omega M_r\ar[r]^{i'} & N_{s+1}P_n^{A_{s+1}}\oplus \cdots \oplus N_rP_n^{A_r} \oplus P''}$$
This all boils down to the fact that if we have a linear combination like (\ref{eq:lin_comb_A_i}), then the same combination is obtained in $\modu\,\Lambda_{\mathcal{M}}$ for the next syzygy level, more precisely:
\begin{align*}
	& N_1\Omega^{n+1}(A_1,Q_1,f_1,g_1) \oplus \cdots \oplus N_s\Omega^{n+1}(A_s,Q_s,f_s,g_s) \oplus (0,M\otimes_AP',0,0) \simeq \\
	& \simeq N_{s+1}\Omega^{n+1}(A_{s+1},Q_{s+1},f_{s+1},g_{s+1}) \oplus \cdots \oplus (0,M\otimes_AP'',0,0).
\end{align*}
Writing the previous isomorphism in terms of classes we obtain the desired linear combination:
\begin{align*}
	& N_1[\Omega^{n+1}(A_1,Q_1,f_1,g_1)] + \cdots + N_s[\Omega^{n+1}(A_s,Q_s,f_s,g_s)] = \\
	& = N_{s+1}[\Omega^{n+1}(A_{s+1},Q_{s+1},f_{s+1},g_{s+1})] + \cdots + N_r[\Omega^{n+1}(A_r,Q_r,f_r,g_r)].
\end{align*}
Note that we have used Lemma \ref{Syzygy_triang} to calculate the syzygies in $\modu\,\Lambda_{\mathcal{M}}$.
\end{proof}

\begin{pro}\label{pro:syzygy_n+1}
Let $(A,M\otimes_TP,f,g)\in\modu\,\Lambda_{\mathcal{M}}$ such that $P\in\proj\,T$. Then, $$\Phi(A,M\otimes_TP,f,g)\leq \Phi(A)+1.$$
\end{pro}

\begin{proof}
Suppose that $\Phi(A,M\otimes_TP,f,g)=n>\Phi(A)+1$, therefore $n\geq 2$. Then we have a linear combination like that of (\ref{eq:lin_comb_A_i}) that derives from the linear combination in $\modu\,\Lambda_{\mathcal{M}}$, since the rank at the $n$ syzygy level has to decrease. Since $\Phi(A)\leq n-2$ the relation has to be maintained at the level $n-2$, but if this is the case, by Lemma \ref{lem:relation_A_and_(A,P,f)} the same relation can be found at level $n-1$ in $\modu\,\Lambda_{\mathcal{M}}$, which contradicts that $\Phi(A,M\otimes_TP,f,g)=n$. In conclusion $\Phi(A,M\otimes_TP,f,g)\leq \Phi(A)+1.$ 
\end{proof}
 
The following corollary follows from the proof of the previous proposition.

\begin{cor}\label{cor:Phi(A,0,0)}
	Let $(A,0,0,0)\in\modu\,\Lambda_{\mathcal{M}}$. Then, $$\Phi(A)\leq \Phi(A,0,0,0)\leq \Phi(A)+1.$$
\end{cor}

Because $M$ and $N$ have symmetric roles in $\Lambda_{\mathcal{M}}$, we can obtain symmetric results in the following sense:

\begin{lem}\label{lem:relation_B_and_(Q,B,f,g)}
	Let $(N\otimes_UQ,B,f,g)\in\modu\,\Lambda_{\mathcal{M}}$ such that $Q\in\proj\,U$. Denote the set of indecomposable, pairwise non isomorphic, direct summands of $(N\otimes_UQ,B,f,g)$ by $\{(Q_i,B_i,f_i,g_i)\}_{i=1}^r$. If for some $n\geq 0$ and $L_i\in\mathbb{Z}_{\geq 0}, \forall i=1,\cdots, r$ we have a linear combination of the form:
	$$L_1[\Omega^nB_1] + \cdots + L_s[\Omega^nB_s] = L_{s+1}[\Omega^nB_{s+1}] + \cdots + L_r[\Omega^nB_r],$$
	then there is also a linear combination of the form:
	\begin{align*}
		& L_1[\Omega^{n+1}(Q_1,B_1,f_1,g_1)] + \cdots + L_s[\Omega^{n+1}(Q_s,B_s,f_s,g_s)] = \\
		& = L_{s+1}[\Omega^{n+1}(Q_{s+1},B_{s+1},f_{s+1},g_{s+1})] + \cdots + L_r[\Omega^{n+1}(Q_r,B_r,f_r,g_r)].
	\end{align*}
\end{lem}

\begin{pro}\label{pro:syzygy_B_n+1}
	Let $(N\otimes_UQ,B,f,g)\in\modu\,\Lambda_{\mathcal{M}}$ such that $Q\in\proj\,U$. Then, $$\Phi(N\otimes_UQ,B,f,g)\leq \Phi(B)+1.$$
\end{pro}

\begin{cor}\label{cor:Phi(0,B,0,0)}
	Let $(0,B,0,0)\in\modu\,\Lambda_{\mathcal{M}}$. Then, $$\Phi(B)\leq \Phi(0,B,0,0)\leq \Phi(B)+1.$$
\end{cor}

We leave the proofs to the reader as they are analogous to those of Lemma \ref{lem:relation_A_and_(A,P,f)}, Proposition \ref{pro:syzygy_n+1} and Corollary \ref{cor:Phi(A,0,0)}.\\

Putting together Corollaries \ref{cor:Phi(A,0,0)} and \ref{cor:Phi(0,B,0,0)} we have the following relation between the $\Phi$-dimension of $T$, $U$, $\Lambda_{\mathcal{M}}$.

\begin{cor}\label{key}
	Let $\mathcal{M}=(T, N, M, U, \alpha, \beta)$ be a Morita context such that $_{U}M, M_T,$ $_{T}N, N_U$ are projective modules and $M\otimes_T N = N\otimes_U M =0.$ Then, 
	\begin{align*}
		\max\{\Phidim(T),\Phidim(U)\} \leq \Phidim(\Lambda_{\mathcal{M}}).
	\end{align*}
	where $\Lambda_{\mathcal{M}}$ is the Morita algebra associated to $\mathcal{M}$.
\end{cor}

Next we give an example illustrating Lemma \ref{lem:relation_A_and_(A,P,f)} as well as Proposition \ref{pro:syzygy_n+1} and Corollary \ref{cor:Phi(A,0,0)}.

\begin{ex} \label{ej:}
Let $T=\frac{\K Q}{J^2}$ be the $\K$-algebra given by the following quiver 
$$\begin{tikzcd}
    Q: & 1\ar[loop, out=120, in=50, distance=2em]{}\ar{r} & 2\ar{r} & 3,
\end{tikzcd}$$
where $J$ is the ideal generated by the arrows. Consider the algebra $\Lambda=\begin{pmatrix}
T & 0 \\
T & T
\end{pmatrix}$, in what follows we will see that for $M:=S_1\oplus I_2\in\modu\,T$ we have $\Phi(M)=2$ and $\Phi(M,0,0)=3$. This example shows that the inequality of Corollary \ref{cor:Phi(A,0,0)} can be strict and that the inequality of Proposition \ref{pro:syzygy_n+1} cannot be refined. In order to do so let us write the minimal projective resolutions for $S_1$ and $I_2$:
$$\begin{tikzcd}[column sep=tiny,row sep=small]
    \cdots \ar{r} & P_1\oplus P_2\oplus P_3 \ar{rr} \ar{dr} & & P_1\oplus P_2 \ar{rr} \ar{dr} & & P_1 \ar{r} & S_1 \ar{r} & 0
    \\
     & & S_1\oplus S_2\oplus S_3 \ar{ur} & & S_1\oplus S_2 \ar{ur}
\end{tikzcd}$$
$$\ \ \ \ \ \ \ \ \ \ \ \ \ \ \begin{tikzcd}[column sep=small,row sep=small]
    \cdots \ar{r} & P_1\oplus P_2 \ar{rrr} \ar{dr} & & & P_1 \ar{dr} \ar{rr} & & P_1 \ar{r} & I_2 \ar{r} & 0
    \\
     && S_1\oplus S_2 \ar{urr} & & & S_1 \ar{ur}
\end{tikzcd}$$
We can see here that
$$\rk\left\langle [\Omega S_1], [\Omega I_2] \right\rangle = \rk\left\langle [S_1] + [S_2], [S_1] \right\rangle=2.$$
Since $S_3=P_3$ is projective, $$\rk\left\langle [\Omega^2S_1], [\Omega^2I_2] \right\rangle = \rk\left\langle [S_1] + [S_2] + [S_3], [S_1] + [S_2] \right\rangle=1.$$
Furthermore $$\rk\left\langle [\Omega^mS_1], [\Omega^mI_2]\right\rangle = 1, \forall m\geq 2, $$ so we get that $\Phi(S_1\oplus I_2)=2.$ 

Let us now calculate $\Phi(S_1\oplus I_2,0,0)$. The projective resolutions are as follows:
$$\begin{tikzcd}[column sep=tiny,row sep=small]
    \cdots \ar{r} & Q_2 \ar{rr} \ar{dr} & & Q_1 \ar{rr} \ar{dr} & & Q_0 \ar{r} & (S_1,0,0) \ar{r} & 0
    \\
     & & (\Omega^2S_1,P_1\oplus P_2, i)\ar[ru] & & (\Omega S_1,P_1,i) \ar{ur}
\end{tikzcd}$$
$$\begin{tikzcd}[column sep=tiny,row sep=small]
    \cdots \ar{r} & Q'_2 \ar{rr} \ar{dr} & & Q'_1 \ar{rr} \ar{dr} & & Q'_0 \ar{r} & (I_2,0,0) \ar{r} & 0
    \\
     & & (\Omega^2I_2,P_1, i)\ar[ru] & & (\Omega I_2,P_1,i) \ar{ur}
\end{tikzcd}$$
Here we have used Lemma \ref{Syzygy_triang} to compute syzygies in $\modu\,\Lambda.$ From the previous calculations we know there is a relation between the first components of the second syzygies, meaning $[\Omega^2S_1]=[\Omega^2I_2]$ but as we will see now, that relation cannot be extended to the whole module, i.e. $[(\Omega^2S_1,P_1\oplus P_2, i)]\neq [(\Omega^2I_2,P_1, i)]$. The only way that this classes can be the same is if there are modules $P',P''\in\proj\,\Lambda$ such that 
$$(\Omega^2S_1,P_1\oplus P_2, i)\oplus P' \simeq (\Omega^2I_2, P_1, i) \oplus P''.$$ By the relation taking place in the first components and in order to make the second components isomorphic, there is only one possibility for $P'$ and $P''$, namely $P'=(0,P_3,0)$ and $P''=(P_3,P_3,1)\oplus (0,P_2,0).$ Nevertheless it is easy to check that there is no isomorphism in $\modu\,\Lambda$ for the given triples, even though the components separately are isomorphic. Hence, for any $P', P''\in\proj\,\Lambda$ we have
$$(\Omega^2S_1,P_1\oplus P_2, i)\oplus P' \not\simeq (\Omega^2I_2, P_1, i) \oplus P''.$$
If we go one more step in the syzygy calculations we will see that $$\Omega^3(S_1,0,0) \simeq \Omega^3(I_2,0,0) \oplus (0,P_3,0),$$ which is exactly the same relation the first coordinates had in the previous step. This illustrates Lemma \ref{lem:relation_A_and_(A,P,f)} and implies that 
$$\rk\left\langle [\Omega^3(S_1,0,0)], [\Omega^3(I_2,0,0)] \right\rangle = \rk\left\langle [\Omega^m(S_1,0,0)], [\Omega^m(I_2,0,0)] \right\rangle =1, \forall m\geq 3;$$  which means that $\Phi(M,0,0)=3.$
\end{ex}

In the next lemma we take one step further in studying the relationship between the value of the $\Phi$ function in a 4-tuple representing a $\Lambda_{\M}$ module and the corresponding values on the coordinates.

\begin{lem}\label{lem:(A,P,f)_plus_(0,B,0)}
Let $(A,M\otimes_TP,f,g)\ \text{and} \ (N\otimes_UQ,B,f',g')\in\modu\,\Lambda_{\mathcal{M}}$ such that $P\in\proj\,T$ and $Q\in\proj\,U$. Then, $$\Phi\left( (A,M\otimes_TP,f,g)\oplus (N\otimes_UQ,B,f',g') \right) \leq \max\{\Phi(A),\Phi(B)\}+1.$$
\end{lem}

\begin{proof}
Assume, to the contrary, that $$\Phi\left( (A,M\otimes_TP,f,g)\oplus (N\otimes_UQ,B,f',g') \right)=n > \max\{\Phi(A),\Phi(B)\}+1,$$ which implies $n\geq 2$. Let $\{(A_i,Q_i,f_i,g_i)\}_{i=1}^r$ and $\{(Q_j',B_j,f_j',g_j')\}_{j=1}^{r'}$ be complete sets of indecomposables, pairwise non isomorphic, direct summands of $(A,M\otimes_TP,f,g)$ and $(N\otimes_UQ,B,f',g')$ respectively. Then there is a linear combination of the form:
$$\sum_{i=1}^rN_i[(\Omega^nA_i,M\otimes_TP_{n-1}^{A_i},1\otimes i,0)] + \sum_{j=1}^{r'}R_j[(N\otimes_UP_{n-1}^{B_j},\Omega^nB_j,0,1\otimes i)]=0.$$

In particular, since $M\otimes_TP_{n-1}^{A_i}$ and $N\otimes_UP_{n-1}^{B_j}$ are projective for all $i,j$, there are relations of the form 
\begin{equation}\label{eq:realtion_n}
\sum_{i=1}^rN_i[\Omega^nA_i]=0; \ \ \ \sum_{j=1}^{r'}R_j[\Omega^nB_j]=0.
\end{equation}
Because of our initial assumption, those relations had to be present at the $n-2$ level of syzygy, which is impossible because, from Lemmas \ref{lem:relation_A_and_(A,P,f)} and \ref{lem:relation_B_and_(Q,B,f,g)}, we would obtain a relation of the form 
$$\sum_{i=1}^rN_i[(\Omega^{n-1}A_i,M\otimes_TP_{n-2}^{A_i},1\otimes i,0)] + \sum_{j=1}^{r'}R_j[(N\otimes_UP_{n-2}^{B_j},\Omega^{n-1}B_j,0,1\otimes i)]=0,$$ which contradicts the fact that $\Phi\left( (A,M\otimes_TP,f,g)\oplus (N\otimes_UQ,B,f',g') \right) =n$.
\end{proof}

All this leads to the next theorem which gives us a full picture of the relation between the values of $\Phi$ in $\modu\,\Lambda_{\mathcal{M}}$, $\modu\,T$ and $\modu\,U.$

\begin{teo}\label{teo:inequality_Phi_triangular}
Let $(A,B,f,g)\in\modu\,\Lambda_{\mathcal{M}}$. Then $$\Phi(A,B,f,g)\leq \max\{\Phi(\Omega A),\Phi(\Omega B)\}+2.$$
\end{teo}

\begin{proof}
From the properties of the Igusa-Todorov functions we know that
\begin{align*}
	\Phi(A,B,f,g) & \leq \Phi\left( \Omega(A,B,f,g) \right) + 1 \\
	& = \Phi\left( (\Omega A,M\otimes_TP_0^A,1\otimes i,0)\oplus (N\otimes_UP_0^B,\Omega B,0,1\otimes i)\right)+1.
\end{align*} 

 If we apply now Lemma \ref{lem:(A,P,f)_plus_(0,B,0)} we get $$\Phi(A,B,f,g)\leq \max\{\Phi(\Omega A),\Phi(\Omega B)\}+2.$$
\end{proof}

In the particular case the Morita algebra is a finite dimensional $k$-algebra, the authors in \cite[Prop 3.5]{BM4} give sufficient conditions for the Morita algebra to have finite $\Phi$-dimension in terms of the quiver with relations. As a consequence of the previous results we obtain the following bounds involving the $\Phi$-dimension of a Morita algebra and the $\Phi$-dimension of the diagonal subalgebra. We want to emphasize that the conditions in \cite{BM4} and  ours differ in many ways, for instance we work with Artin algebras instead of finite dimensional algebras and their hypotheses imply the bimodules in the Morita context are semisimple as right modules, while in our case we ask that they are projective.

\begin{cor}\label{cor:phi_dim_Lambda_T_U}
	Let $\mathcal{M}=(T, N, M, U, \alpha, \beta)$ be a Morita context such that $_{U}M, M_T,$ $_{T}N, N_U$ are projective modules and $M\otimes_T N = N\otimes_U M =0.$ Then, 
    \begin{align*}
    	\Phidim(\Lambda_{\mathcal{M}}) & \leq \max\{\Phidim(\Omega(\modu\,T)),\Phidim(\Omega(\modu\,U))\}+2 \\
    	& \leq \max\{\Phidim(T),\Phidim(U)\}+2,
    \end{align*}
    where $\Lambda_{\mathcal{M}}$ is the Morita algebra associated to $\mathcal{M}$.
\end{cor}

The next theorem gives us a full picture of the relation between the $\Phi$-dimensions of $T$, $U$ and $\Lambda_{\mathcal{M}}$.

\begin{teo}\label{teo:Phi_Dim_Morita}
	Let $\mathcal{M}=(T, N, M, U, \alpha, \beta)$ be a Morita context such that $_{U}M, M_T,$ $_{T}N, N_U$ are projective modules and $M\otimes_T N = N\otimes_U M =0.$ Then, 
	\begin{align*}
		\max\{\Phidim(T),\Phidim(U)\} \leq \Phidim(\Lambda_{\mathcal{M}})\leq \max\{\Phidim(T),\Phidim(U)\} + 2.
	\end{align*}
	where $\Lambda_{\mathcal{M}}$ is the Morita algebra associated to $\mathcal{M}$.
\end{teo}

The following theorem gives us, under some assumptions on the Morita context, a necessary and sufficient condition for a Morita algebra to be GLIT. We recall that the families of IT, LIT and GLIT algebras are related by proper inclusion, so GLIT algebra is the more general concept. For the other two definitions the problem of when the Morita algebra also satisfy them has been studied in \cite[Theorem 1.2]{IT-Morita} and \cite[Prop 3.2 and 3.3]{BM3}.

\begin{teo}\label{teo:Morita_GLIT}
Let $\mathcal{M}=(T, N, M, U, \alpha, \beta)$ be a Morita context such that $_{U}M, M_T,$ $_{T}N, N_U$ are projective modules and $M\otimes_T N = N\otimes_U M =0.$ Then, the algebra $\Lambda_{\mathcal{M}}=\begin{pmatrix}
T & N \\
M & U
\end{pmatrix}$ is GLIT if and only if $T$ and $U$ are GLIT.

\end{teo}

\begin{proof}
($\Leftarrow$) We first prove that if $T$ and $U$ are GLIT, then $\Lambda_{\mathcal{M}}$ is GLIT. Assume that $T$ is $(n,t,V,\T)$-GLIT and $U$ is $(n,u,W,\U)$-GLIT, observe that, because of Proposition \ref{pro:GLIT_bigger_integers}, there is no loss of generality in assuming that both have the same parameter $n\geq 0$.\\
We will show that $\Lambda_{\mathcal{M}}$ is $(n+1,\lambda,Z,\D_{\Lambda}^{\mathcal{M}})$-GLIT, where $\lambda=t\cdot u$, $Z=\Omega(V,W,0)\oplus \Lambda_{\mathcal{M}}$ and $\D_{\Lambda}^{\mathcal{M}}=\add(\Omega(\T,\U,0,0))$. The first step is to build, for any $(A,B,f,g)\in\modu\,\Lambda_{\mathcal{M}}$, a short exact sequence as in Definition \ref{defi:GLIT_class} and after doing so we will see that the parameters we have fixed satisfy the conditions needed. Let us recall item (1) of Remark \ref{rmk:Syzygy_triang} stating that for any $n\geq 1$ we have $$\Omega^n(A,B,f,g) = (\Omega^n(A),M\otimes P_{n-1}^A,1\otimes i_n,0) \oplus (N\otimes_UP_{n-1}^B, \Omega^n(B),0,1\otimes i_n).$$
The strategy then is to construct a s.e.s for each summand separately and then combine them via direct sums.\\

Since $T$ and $U$ are $n$-GLIT algebras, there are short exact sequences as these:
$$\xymatrix{0\ar[r] & T_1\oplus V_1\ar[r]^{\alpha} & T_0\oplus V_0\ar[r]^{\beta} & \Omega^n(A)\ar[r] & 0 },$$ where $T_1,T_0\in \T,\ V_1,V_0\in \add\,V;$

$$\xymatrix{0\ar[r] & U_1\oplus W_1\ar[r]^{\alpha'} & U_0\oplus W_0\ar[r]^{\beta'} & \Omega^n(B)\ar[r] & 0 },$$ where $U_1,U_0\in \U,\ W_1,W_0\in \add\,W.$\\

Since $M_T$ is projective, the following diagram is commutative with exact rows:
$$\xymatrix{
0\ar[r] & M\otimes_T (T_1\oplus V_1)\ar[r]^{1\otimes\alpha}\ar[d]^{0} & M\otimes_T (T_0\oplus V_0)\ar[d]^{1\otimes i_n\circ\beta}\ar[r]^{1\otimes\beta} & M\otimes_T\Omega^n(A) \ar[d]^{1\otimes i_n}\ar[r] & 0 \\
0\ar[r] & 0\ar[r]^{0} & M\otimes P_{n-1}^A\ar[r]^{id} & M\otimes P_{n-1}^A\ar[r] & 0
}
$$

From this diagram we get the following s.e.s. in $\modu\,\Lambda_{\mathcal{M}}$:
$$\footnotesize\xymatrix{(T_1\oplus V_1,0,0,0)\ar@{*{\dir^{(}}-*{\dir{>}}}[r]^(.35){(\alpha,0)} & (T_0\oplus V_0,M\otimes P_{n-1}^A,1\otimes i_n\circ\beta,0)\ar@{*{\dir{>>}}}[r]^(.52){(\beta,id)} & (\Omega^n(A),M\otimes P_{n-1}^A,1\otimes i_n,0)}$$

In a similar way, since $N_U$ is projective, the following diagram is commutative and has exact rows:

$$\xymatrix{
	0\ar[r] & N\otimes_U (U_1\oplus W_1)\ar[r]^{1\otimes\alpha'}\ar[d]^{0} & N\otimes_U (U_0\oplus W_0)\ar[d]^{1\otimes i_n\circ\beta'}\ar[r]^{1\otimes\beta'} & N\otimes_U\Omega^n(B) \ar[d]^{1\otimes i_n}\ar[r] & 0 \\
	0\ar[r] & 0\ar[r]^{0} & N\otimes P_{n-1}^B\ar[r]^{id} & N\otimes P_{n-1}^B\ar[r] & 0
}
$$
From this diagram we get the following s.e.s. in $\modu\,\Lambda_{\mathcal{M}}$:
$$\footnotesize\xymatrix{(0,U_1\oplus W_1,0,0)\ar@{*{\dir^{(}}-*{\dir{>}}}[r]^(.38){(0,\alpha')} & (N\otimes P_{n-1}^B,U_0\oplus W_0,0,1\otimes i_n\circ\beta')\ar@{*{\dir{>>}}}[r]^(.52){(id,\beta')} & (N\otimes P_{n-1}^B,\Omega^n(B),0,1\otimes i_n)}$$
Combining both sequences via direct sums we obtain the following s.e.s:
$$\delta: \ \xymatrix{X\ar@{*{\dir^{(}}-*{\dir{>}}}[r] & Y \ar@{*{\dir{>>}}}[r] & \Omega^n(A,B,f,g)},$$ 
where $Y=(T_0\oplus V_0,M\otimes P_{n-1}^A,1\otimes i_n\circ\beta,0)\oplus (N\otimes P_{n-1}^B,U_0\oplus W_0,0,1\otimes i_n\circ\beta')$ and $X=(T_1\oplus V_1,U_1\oplus W_1,0,0)$.\\
Next, let us apply the Horseshoe Lemma to the sequence $\delta$: 
\begin{equation}\label{eq:s.e.c_delta}
\xymatrix{\Omega X\ar@{*{\dir^{(}}-*{\dir{>}}}[r] & \Omega Y\oplus \tilde{Q} \ar@{*{\dir{>>}}}[r] & \Omega^{n+1}(A,B,f,g)},
\end{equation}
where $\Omega X=\Omega(T_1,U_1,0)\oplus \Omega(V_1,W_1,0)$, $\Omega Y = \Omega(T_0,U_0,0)\oplus \Omega(V_0,W_0,0)$ and $\tilde{Q}\in \proj\,\Lambda_{\mathcal{M}}$.\\

As one can see the terms $\Omega X$ and $\Omega Y\oplus \tilde{Q}$ are in the class $\D_{\Lambda}^{\mathcal{M}} \oplus \add\,Z$, so all we need to check is that the conditions on $\D_{\Lambda}^{\mathcal{M}}$ are satisfied. Using Remark \ref{addD} we only need to prove that $\D'_{\Lambda}:= \Omega(\T,\U,0,0)$ is closed under direct sums, is $\lambda$-syzygy invariant in $\modu\,\Lambda_{\mathcal{M}}$ and has finite $\Phi$-dimension.\\

First, it is straightforward that $\D_{\Lambda}'$ is closed under direct sums. To see that it is $\lambda$-syzygy invariant, take $\Omega(A,B,0,0)\in \D_{\Lambda}'$. Then, 
\begin{align*}
	\Omega^{\lambda}\left( \Omega(A,B,0,0)\right) = \Omega^{\lambda+1} (A,B,0,0) = \Omega (\Omega^{\lambda}(A),\Omega^{\lambda}(B),0,0).
\end{align*}
The last equality follows from Remark \ref{rmk:Syzygy_triang} and the fact that syzygies commute with direct sums. Since $\lambda=t\cdot u$, we have $\Omega^{\lambda}(A)\in\T$ and $\Omega^{\lambda}(B)\in\U.$

Finally to see that the $\Phi$-dimension is finite, take $\Omega(A,B,0,0)\in\D_{\Lambda}'$. Using Proposition \ref{pro:huard} and Lemma \ref{lem:(A,P,f)_plus_(0,B,0)} we get
\begin{align*}
	\Phi\left( \Omega(A,B,0,0)\right) & \leq \Phi\left( \Omega^{\lambda}(A,B,0,0)\right) + \lambda-1 \\ & \leq \max\{\Phi(\Omega^{\lambda}A),\Phi(\Omega^{\lambda}B) \}+\lambda 
	\\ & \leq \max\{\Phidim(\T),\Phidim(\U)\}+\lambda <\infty.
\end{align*}

$(\Rightarrow)$ Now we prove that if $\Lambda_{\mathcal{M}}$ is GLIT then both $T$ and $U$ are GLIT. Assume that $\Lambda_{\mathcal{M}}$ is $(n,\lambda,V_{\Lambda},\D_{\Lambda})$-GLIT, where $V_{\Lambda}=(X,E,f,j)$. We will show that $T$ is $(n,\lambda,X,\T)$-GLIT where $\T=\add\{Y\in\modu\,T| \ \exists G,g_1,g_2: \ (Y,G,g_1,g_2)\in\D_{\Lambda}\}$ and that $U$ is $(n,\lambda,E,\U)$, where $\U=\add\{Z\in\modu\,U| \ \exists H,h_1,h_2: \ (H,Z,h_1,h_2)\in\D_{\Lambda}\}$.\\

We only include the proof for $T$, since the one for $U$ is analogous. Let $A\in\modu\,T$, then $(A,0,0,0)\in\modu\,\Lambda_{\mathcal{M}}$ so there is a short exact sequence of the form 
$$\xymatrix{0\ar[r] & D_1\oplus V_1\ar[r] & D_0\oplus V_0\ar[r] & \Omega^n(A,0,0,0)\ar[r] & 0 },$$ where $D_1,D_0\in \D_{\Lambda},\ V_1,V_0\in \add\,V_{\Lambda}.$ If we restrict to the first components we get the following short exact sequence in $\modu\,T$ 
$$\xymatrix{0\ar[r] & Y_1\oplus X_1\ar[r] & Y_0\oplus X_0\ar[r] & \Omega^n(A)\ar[r] & 0 },$$ where $Y_1,Y_0\in\T$ and $X_1,X_0\in\add\,X.$ Because of Remark \ref{addD} all we need to show to conclude that $T$ is GLIT, is that the class $\T'=\{Y\in\modu\,T| \ \exists G,g_1,g_2: \ (Y,G,g_1,g_2)\in\D_{\Lambda}\}$ is closed under direct sums, is $\lambda$-syzygy invariant and has finite $\Phi$-dimension.\\
It is clear that $\T'$ is closed under direct sums. To see that it is $\lambda$-syzygy invariant, take $Y\in\T'$, since $\D_{\Lambda}$ is $\lambda$-syzygy invariant and closed under direct summands, we get that $\Omega^{\lambda}(Y,0,0,0)\in\D_{\Lambda}$ and this proves that $\Omega^{\lambda}(Y)\in\T'.$ To prove that the $\Phi$-dimension of $\T'$ is finite, take $Y\in\T'$ and by Corollary \ref{cor:Phi(A,0,0)} we know that $\Phi(Y)\leq\Phi(Y,0,0,0)$. Using Proposition \ref{pro:huard} we get $$\Phi(Y)\leq\Phi(Y,0,0,0)\leq\Phi(\Omega^{\lambda}(Y,0,0,0))+\lambda\leq\Phidim(\D_{\Lambda})+\lambda.$$ Hence, $\Phidim(\T')\leq \Phidim(\D_{\Lambda})+\lambda < \infty.$

\end{proof}

We have a similar result for special GLIT algebras.

\begin{teo}\label{teo:Morita_special_GLIT}
	Let $\mathcal{M}=(T, N, M, U, \alpha, \beta)$ be a Morita context such that $_{U}M, M_T,$ $_{T}N, N_U$ are projective modules and $M\otimes_T N = N\otimes_U M =0.$ Then, the algebra $\Lambda_{\mathcal{M}}=\begin{pmatrix}
		T & N \\
		M & U
	\end{pmatrix}$ is special GLIT if and only if $T$ and $U$ are special GLIT.
\end{teo}

\begin{proof}
	First we recall that being special GLIT is the same as having finite finitistic dimension. In addition it is not hard to show that $(A,B,f,g)\in\modu\Lambda_{\mathcal{M}}$ has finite projective dimension if and only if $A$ and $B$ have finite projective dimension. Then, using Theorem \ref{teo:Phi_Dim_Morita}, together with the fact that the $\Phi$ function coincides with projective dimension if the last one is finite, we get: $$\max\{\findim(T),\findim(U)\}\leq \findim(\Lambda_{\mathcal{M}})\leq \max\{\findim(T),\findim(U)\} + 2.$$
	Hence, $\Lambda_{\mathcal{M}}$ is special GLIT if and only if $T$ and $U$ are special GLIT.\\
\end{proof}

\begin{remark}\label{rmk:findim_Morita}
    We recall that special GLIT algebras are exactly those that satisfy the finitistic dimension conjecture, therefore the previous theorem states that $\findim(\Lambda_{\mathcal{M}})<\infty$ if and only if $\findim(T)<\infty$ and $\findim(U)<\infty$. After carefully reviewing the literature we were able to find that the direct implication can also be obtained from \cite[Theorem 5.1 and Theorem 5.5(i)]{Psadu} and also from \cite[Proposition 10]{Kirk}. Note that for the converse implication we cannot use \cite[Theorem 5.5(iii)]{Psadu} since the hypotheses are not satisfied, so as far as we know, there are no published results that prove the converse, which came as a surprise to us, particularly since it was not in the objectives of this paper.
\end{remark}

As a direct consequence of Theorems \ref{teo:Morita_GLIT} and \ref{teo:Morita_special_GLIT} we obtain the following corollary which gives necessary and sufficient conditions for a triangular matrix algebra to be GLIT or special GLIT. The question of when a triangular algebra is IT or LIT has been studied in \cite{BLM} and \cite{Vivero}.

\begin{cor}\label{cor:Triang_GLIT}
	Let $T$ and $U$ be Artin algebras and $_{U}M_T$ a $U$-$T$-bimodule such that $_{U}M$ and $M_T$ are projective. Then, the triangular matrix algebra $\begin{pmatrix}
		T & 0 \\ M & U
	\end{pmatrix}$ is (special)GLIT if and only if $T$ and $U$ are (special)GLIT.
\end{cor}

We finish the paper with an application of the previous result, obtaining a partial answer to the question of weather the tensor product of GLIT algebras is again GLIT.

\begin{pro}\label{pro:tensor_tree}
Let $T$ be a GLIT $\K$-algebra and $Q$ a finite quiver without oriented cycles. Then $T\otimes_{\K}\K Q$ is GLIT.
\end{pro}

\begin{proof}
	From \cite[Lemma 1.12 ]{AssemRT} we know that $$T\otimes_{\K}\K Q=\begin{pmatrix}
	T^{d_{1,1}} & 0 & \cdots & 0 \\
	T^{d_{2,1}} & T^{d_{2,2}} & \cdots & 0 \\
	\vdots & \vdots & \vdots & \vdots\\
	T^{d_{n,1}} & T^{d_{n,2}} & \cdots & T^{d_{n,n}}
	\end{pmatrix},$$ where $d_{i,j}=\dim_{\K}\epsilon_i\K Q\epsilon_j$. Because of the way this isomorphism is defined one can see that this algebra can be thought of as a triangular matrix algebra of the form $$\left( \begin{array}{ccc|c}
		T^{d_{1,1}} & 0 & \cdots & 0 \\
		T^{d_{2,1}} & T^{d_{2,2}} & \cdots & 0 \\
		\vdots & \vdots & \vdots & \vdots\\
		\hline
		T^{d_{n,1}} & T^{d_{n,2}} & \cdots & T^{d_{n,n}}
	\end{array}\right),$$ where the bimodule $M=(T^{d_{n,1}}, T^{d_{n,2}}, \cdots, T^{d_{n,n-1}})$ is free as a left $T^{d_{n,n}}=T$-module and projective as a right module over the top left matrix algebra, since it is a direct sum of certain number of copies of each of the previous rows. Using induction on $n$ we get the desired result by applying Corollary \ref{cor:Triang_GLIT}.
\end{proof}



\bibliographystyle{unsrt}

\begin{thebibliography}{99}
    \bibitem{AssemRT} I. Assem, D. Simson, A. Skowronski. Elements of the Representation Theory of Associative Algebras. \textit{L. M. S. Student Texts 65}, (2006).

    
    


	\bibitem{ARS} M. Auslander, I.  Reiten,  S. O.  Smal$\phi$.
	\newblock Representation theory of Artin algebras
	\newblock {\it Cambridge Studies in Advanced Mathematics, 36. Cambridge University Press, Cambridge} (1995), xiv+423 pp. ISBN: 0-521-41134-3.
	
	
	\bibitem{Aus} M. Auslander. Representation Dimension of Artin Algebras. {\it  Queen Mary College Math. Notes, London} (1971).
	
	\bibitem{Barrios} M. Barrios, G. Mata. On algebras of $\Omega^n$-finite and $\Omega^\infty$-infinite representation type. {\it Journal of Algebra and Its Applications.}  DOI: https://doi.org/10.1142/S0219498824501834.
	
	\bibitem{Barrios2} M. Barrios, G. Mata. On Lat-Igusa-Todorov algebras. {\it  São Paulo Journal of Mathematical Sciences} 
	16 (2022), no. 2, 693-711.
	
	\bibitem{BM3} M. Barrios, G. Mata. Igusa-Todorov and LIT algebras on Morita context algebras.  {\it  arXiv:} 2307.07570 [math.RT].
	
	\bibitem{BM4} M. Barrios, G. Mata. The Igusa-Todorov $\Phi$-dimension on Morita context algebras.  {\it  arXiv:} 2211.06473 [math.RT].
	
	\bibitem{Bass} H. Bass. Finitistic Dimension and a Homological Generalization of Semi-Primary Rings. \emph{Trans. Amer. Math. Soc.} (1960), 95, 466-488.
	
	\bibitem{Bass2} H. Bass. The Morita theorems. \emph{Mimeographed notes, University of Oregon} (1962).
	
	
	
	\bibitem{BLM} D. Bravo, M. Lanzilotta, O. Mendoza. Pullback diagrams, syzygy finite classes and Igusa-Todorov algebras. {\it Journal of Pure and Applied Algebra} Vol 223(10), 4494-4508 (2019).
	
	\bibitem{BLMV} D. Bravo, M. Lanzilotta, O. Mendoza, J. Vivero. Generalised Igusa-Todorov functions and Lat-Igusa-Todorov algebras. {\it  Journal of Algebra} Vol. 580, 63-83, 2021.
	
	\bibitem{Cohn} P.M. Cohn. Morita equivalence and duality. {\it Qeen Mary College Math. Notes} (1966).
	
	\bibitem{TC} T. Conde. On certain strongly quasihereditary algebras. {\it University of Oxford} (2015), Ph.D. Thesis.
	
	
    
    \bibitem{Idun} R. M. Fossum, P. A. Griffith, I. Reiten. Trivial extensions of
    abelian categories. \emph{ Lecture Notes in Mathematics. Springer-Verlag, Berlin-New York} (1975), Vol. 456.
    
    \bibitem{Psa} E. L. Green, C. Psaroudakis.
    \newblock On Artin algebras arising from Morita contexts.
    \newblock \emph{Algebr. Represent. Theory} (2014), 17(5), 1485?1525.
    
    \bibitem{Hanson} E. Hanson, K. Igusa. A Counterexample of the $\phi$-Dimension Conjecture. {\it  arXiv:} 1911.00614v1 [math.RT] (2019).

	
	\bibitem{Huang_Sun} Z. Huang, J. Sun. Endomorphism algebras and Igusa-Todorov algebras. \emph{ActaMath. Hung.} (2013), 140(1-2), 60-70.
	
	\bibitem{HL} F. Huard, M. Lanzilotta.
	\newblock Self-injective right Artinian rings and Igusa-Todorov functions.
	\newblock \emph{Algebr. Represent. Theory} (2013), 16, 765-770.
	
	\bibitem{HLM1} F. Huard, M. Lanzilotta, O. Mendoza. An approach to the finitistic dimension conjecture. \emph{Journal of Algebra} (2008), 319, 3918-3934.
	
	
	\bibitem{IT} K. Igusa,  G. Todorov. On the finitistic global dimension conjecture for Artin algebras. \emph{Representations of algebras and related topics} (2005), 201-204, Fields Inst. Commun., 45, Amer. Math. Soc., Providence, RI.
	
	\bibitem{Kirk} E. Kirkman, J. Kuzmanovich.
	\newblock On the finitistic dimension of fixed subrings.
	\newblock \emph{Communications in Algebra} (1994), 22, 4621-4635.
	
	
	
	 
    
    \bibitem{IT-Morita} Y. Ma, N. Ding.
    \newblock Igusa-Todorov algebras over Morita rings.
    \newblock \emph{Journal of Algebra and Its Applications.} Vol. 22, No. 08, 2350178 (2023).
	 
	 
	 \bibitem{Psadu} C. Psaroudakis.
	 \newblock Homological theory of recollements of abelian categories.
	 \newblock \emph{Journal of Algebra.} Vol. 398, 63?110, (2014).
	 
	 \bibitem{R} R. Rouquier. Representation dimension of exterior algebras. \emph{Inventiones mathematicae} (2006), 165, 357-367.
	 
	
	 
	 
	
	\bibitem{Vivero} J. Vivero.
	\newblock Triangular Lat-Igusa-Todorov algebras.
	\newblock \emph{Abh. Math. Semin. Univ. Hambg.} 92, 53-67 (2022). https://doi.org/10.1007/s12188-022-00257-3.
	
	\bibitem{W} J. Wei.
	\newblock Finitistic dimension and Igusa-Todorov algebras.
	\newblock \emph{Adv. Math.} (2009), 222, no. 6, 2215-2226.
	
	
	\bibitem{Z} B. Zimmermann-Huisgen. The finitistic dimension conjectures a tale of 3.5  decades. Abelian groups and modules. \emph{Math. Appl.} (1995), 343, \emph{ Kluwer Acad. Publ., Dordrecht,} 501-517.
\end{thebibliography}

\begin{center}

\end{center}

\end{document}